\documentclass[11pt]{amsart}
\usepackage{amsfonts}
\usepackage[all,cmtip]{xy}
\usepackage{amsmath}
\usepackage{amsthm}
\usepackage{amssymb}
\usepackage[bookmarks, colorlinks=true, linkcolor=blue, citecolor=blue, urlcolor=blue]{hyperref}

\usepackage{anysize}
\marginsize{1in}{1in}{1in}{1in}

\usepackage{color}
\definecolor{prpl}{rgb}{0.7, 0.0, 0.7}

\newcommand{\arxiv}[1]{\href{http://arxiv.org/abs/#1}{{\tt arXiv:#1}}}

\newtheorem{theorem}[equation]{Theorem}

\newtheorem{lemma}[equation]{Lemma}
\newtheorem{corollary}[equation]{Corollary}
\newtheorem{conjecture}[equation]{Conjecture}

\theoremstyle{definition}
\newtheorem{rmk}[equation]{Remark}
\newenvironment{remark}[1][]{\begin{rmk}[#1] \pushQED{\qed}}{\popQED \end{rmk}}
\newtheorem{eg}[equation]{Example}

\newtheorem{defn}[equation]{Definition}

\newcommand{\lw}{{\textstyle \bigwedge}}

\DeclareMathOperator{\Sym}{Sym}
\DeclareMathOperator{\Tor}{Tor}

\newcommand{\bC}{\mathbf{C}}

\newcommand{\bE}{\mathbf{E}}

\newcommand{\cP}{\mathcal{P}}

\newcommand{\bR}{\mathbf{R}}

\newcommand{\bS}{\mathbf{S}}

\newcommand{\bT}{\mathbf{T}}

\newcommand{\bV}{\mathbf{V}}

\title{A remark on a conjecture of Derksen}
\author{Andrew Snowden}
\date{\today}
\thanks{The author was supported by NSF fellowship DMS-0902661.}

\begin{document}

\begin{abstract}
Let $V$ be a complex representation of a finite group $G$ of order $g$.  Derksen conjectured that the $p$th syzygies of the invariant ring $\Sym(V)^G$ are generated in degrees $\le (p+1)g$.  We point out that a simple application of the theory of twisted commutative algebras --- using an idea due to Weyl --- gives the bound $pg^3$.
\end{abstract}

\maketitle

Fix a finite group $G$ of order $g$.  Let $V$ be a finite dimensional complex representation of $G$ and put $R=\Sym(V)^G$, which we regard as a graded ring.  Let $E \subset R$ be a homogeneous vector subspace generating $R$ as an algebra, and put $S=\Sym(E)$, so that $S \to R$ is a surjection of graded algebras.  The space $\Tor^S_p(R, \bC)$ is then a graded vector space, called the space of $p$-syzygies of $R$.  We let $s_p(V; E)$ be the maximal degree occurring in it.  One can show that if $E \subset E'$ then $s_p(V; E) \le s_p(V; E')$.  Furthermore, $s_p(V; E)$ is independent of $E$ if $E$ is chosen to be minimal; we denote this common value by $s_p(V)$.  We then have the following conjecture of Derksen (see \cite[Conj.~3]{Derksen} for a more precise version):

\begin{conjecture}
We have $s_p(V) \le (p+1)g$ for any $V$.
\end{conjecture}

Derksen \cite[Thm.~2]{Derksen} proved the conjecture for $p=1$, but the general case is open.  To state our main result, we first introduce some notation.  Let $\beta(V)$ be the minimal integer such that $R$ is generated in degrees $\le \beta(V)$.  Let $\beta$ be the maximum value of $\beta(V)$ over all $V$, the so-called Noether number of $G$.  Noether's theorem \cite[\S 3]{Wehlau} states that $\beta \le g$.  Let $d_1, \ldots, d_n$ be the degrees of the irreducible representations of $G$ and let $m$ be the sum of the $d_i$.  Finally, put $\delta_p=(\beta-1) g-(m-1) \beta p$; note that this is negative for $p \gg 0$.

\begin{theorem}
\label{mainthm}
We have $s_p(V) \le \beta^2mp+\delta_p$ for any $V$.
\end{theorem}

Although this is weaker than the conjecture, the bound is significant since it is independent of $V$.  Using the fact that $\beta$ and $m$ are both bounded by $g$, we deduce the following corollary:

\begin{corollary}
We have $s_p(V) \le pg^3$ for any $V$.
\end{corollary}

\begin{remark}
In fact, $m$ and $\beta$ are often strictly smaller than $g$, sometimes significantly so.  We have $m \le \sqrt{ng}$, where $n$ is the number of conjugacy classes in $G$.  Typically, $n$ is much smaller than $g$; for example, if $G$ is a symmetric group then $n=O(g^{\epsilon})$, for any $\epsilon>0$.  The currently known bounds on $\beta$ are not very sharp.  A result of Cziszter--Domokos \cite{CziszterDomokos} states that $\beta<\tfrac{1}{2} g$ unless $G$ has a cyclic subgroup of index two, or is one of four exceptions.  A conjecture of Pawale \cite[Conj.~3.9]{Wehlau} asserts that if $G$ is the semi-direct product of cyclic groups of prime orders $p$ and $q$, with $q \mid p-1$, then $\beta=p+q-1$; note that $m=p+q-1$ in this case as well.  When $p$ and $q$ are approximately equal, this gives $m=\beta=O(\sqrt{g})$.  Thus, in many case, our theorem gives a bound of the form $s_p(V)=O(pg^{\theta})$ with $\theta<3$.
\end{remark}

\begin{remark}
Theorem~\ref{mainthm} (and our proof of it) is valid over any field of characteristic zero. Derksen also proposed his conjecture in positive characteristic not dividing $g$.  It would be interesting if our proof could be adapted to work in this setting.
\end{remark}

We now prove the theorem.  Put $s'_p(V)=s_p(V; E)$ where $E=\bigoplus_{i=1}^{\beta} R_i$.  As discussed, we have $s_p(V) \le s'_p(V)$, so it suffices to bound the latter.  The key result is the following lemma.  We make critical use of the theory of Schur functors in its proof.  We also make use of ideas from the theory of twisted commutative algebras, though these objects do not appear by name.  We refer to \cite{expos} for background on this material.

\begin{lemma}
\label{keylem}
For each $p \ge 1$ there is a representation $W_p$ of $G$ such that $s'_p(V) \le s'_p(W_p)$ for any representation $V$.  The dimension of $W_p$ is $\beta mp+g$.
\end{lemma}

\begin{proof}
Let $V_1, \ldots, V_n$ be the irreducible representations of $G$.  For vector spaces $U_1, \ldots, U_n$, put
\begin{displaymath}
\bV=\bV(U_{\bullet})=(V_1 \otimes U_1) \oplus \cdots \oplus (V_n \otimes U_n).
\end{displaymath}
Note that every representation of $G$ is isomorphic to one of the form $\bV(U_{\bullet})$, for some choice of $U_{\bullet}$.  Put $\bR=\Sym(\bV)^G$, let $\bE=\bigoplus_{i=1}^{\beta} \bR_i$, let $\bS=\Sym(\bE)$ and let $\bT=\Tor^{\bS}_p(\bR, \bC)$.  We regard all of these objects as polynomial functors of the $U_i$; as such, each decomposes into Schur functors \cite[\S 6.2.1]{expos}.  The $\bS$-module $\bC$ admits the Koszul resolution $\bS \otimes \lw^{\bullet}(\bE)$, which is functorial (this is explained in the introduction of \cite{expos}).  We can compute $\bT$ by tensoring this resolution with $\bR$.  It follows that $\bT$, as a functor, is a subquotient of $\bR \otimes \lw^p(\bE)$.

For a functor $F$ of the sort we are considering, let $\cP_i(F)$ be the set of partitions $\lambda$ such that the Schur functor $\bS_{\lambda}(U_i)$ appears in $F$ and let $\ell_i(F)$ be the maximum number of rows of any $\lambda \in \cP_i(F)$.  The Cauchy formula \cite[\S 8.2.3]{expos} shows that $\ell_i(\bR) \le d_i$.  Since every partition in $\cP_i(\bE)$ has at most $\beta$ boxes, every partition in $\cP_i(\lw^p(\bE))$ has at most $\beta p$ boxes, and so $\ell_i(\lw^p(\bE)) \le \beta p$.  The Littlewood--Richardson rule \cite[\S 9.1.1]{expos} now gives $\ell_i(\bR \otimes \lw^p(\bE)) \le \beta p+d_i$, and so $\ell_i(\bT) \le \beta p+d_i$.

Let $\bT_k$ be the degree $k$ piece of $\bT$.  Obviously $\ell_i(\bT_k) \le \beta p+d_i$.  It follows that $\bT_k$ is non-zero if and only if $\bT_k(U'_{\bullet})$ is non-zero, where $U'_i=\bC^{\beta p+d_i}$.  We thus find that $\bT(U_{\bullet})$ is concentrated in degrees $\le k$ for all $U_{\bullet}$ if and only if it is so when $U_{\bullet}=U'_{\bullet}$.  We therefore obtain the theorem, with $W_p=\bV(U'_{\bullet})$.  We have
\begin{displaymath}
\dim(W_p)=\sum_{i=1}^n d_i(\beta p+d_i)=\beta mp+g. \qedhere
\end{displaymath}
\end{proof}

\begin{lemma}
We have $s_p'(V) \le (\beta-1)\dim(V)+\beta p$, for any $V$.
\end{lemma}

\begin{proof}
This follows from the proof of \cite[Thm.~1]{Derksen}.
\end{proof}

Combining the two lemmas, we find
\begin{displaymath}
s'_p(V) \le s'_p(W_p) \le (\beta-1)(\beta mp+g)+\beta p= \beta^2mp+ \delta_p
\end{displaymath}
for any $V$, and so the theorem follows.  As $m$ and $\beta$ are both at most $g$, we have
\begin{displaymath}
(\beta-1)(\beta mp+g)+\beta p \le (g-1)(g^2p+g)+gp = pg^3-g(pg+1-p-g) \le pg^3,
\end{displaymath}
where we have used the inequality $p+g \le pg+1$.  This establishes the corollary.

\begin{remark}
The main idea in Lemma~\ref{keylem} --- bound the number of rows in the universal case and then specialize to a particular case of sufficient dimension --- is due to Weyl, who used it to show that $\beta=\beta(\bC[G])$ \cite[Thm.~3.13]{Wehlau}.  The theory of twisted commutative algebras (of which $\bR$ and $\bS$ are examples) systematizes this idea.
\end{remark}

\subsection*{Acknowledgements}
We thank Steven Sam for helpful comments.

\end{document}